\documentclass[10pt]{article}
\font\smallit=cmti10

\usepackage{color}
\usepackage{hyperref}
\usepackage{color}
\usepackage{breakurl}
\usepackage{comment}
\newcommand{\bburl}[1]{\textcolor{blue}{\url{#1}}}
\newcommand\blfootnote[1]{%
  \begingroup
  \renewcommand\thefootnote{}\footnote{#1}%
  \addtocounter{footnote}{-1}%
  \endgroup
}
\usepackage{amssymb,latexsym,amsmath,epsfig,amsthm} 

\makeatletter

\renewcommand\section{\@startsection {section}{1}{\z@}
{-30pt \@plus -1ex \@minus -.2ex}
{2.3ex \@plus.2ex}
{\normalfont\normalsize\bfseries\boldmath}}

\renewcommand\subsection{\@startsection{subsection}{2}{\z@}
{-3.25ex\@plus -1ex \@minus -.2ex}
{1.5ex \@plus .2ex}
{\normalfont\normalsize\bfseries\boldmath}}

\renewcommand{\@seccntformat}[1]{\csname the#1\endcsname. }

\makeatother

\newtheorem{thm}{Theorem}[section]

\newtheorem{cor}[thm]{Corollary}
\newtheorem{lem}[thm]{Lemma}

\newtheorem{exa}[thm]{Example}

\newtheorem{rek}[thm]{Remark}

\newcommand{\Mod}[1]{\ \mathrm{mod}\ #1}

\begin{document}

\begin{center}
\uppercase{\bf Representation of $\mathbf{\frac{1}{2}(F_n-1)(F_{n+1}-1)}$ and $\mathbf{\frac{1}{2}(F_n-1)(F_{n+2}-1)}$}
\vskip 20pt
{\bf H\`ung Vi\d{\^e}t Chu}\\
{\smallit Department of Mathematics, University of Illinois at Urbana-Champaign, Urbana, IL 61820, USA}\\
{\tt chuh19@mail.wlu.edu}\\
\end{center}
\vskip 20pt

\centerline{\bf Abstract}

\noindent Let $a, b\in \mathbb{N}$ be relatively prime. We consider $(a-1)(b-1)/2$, which arises in the study of the $pq$-th cyclotomic polynomial, where $p,q$ are distinct primes. We prove two possible representations of $(a-1)(b-1)/2$ as nonnegative, integral linear combinations of $a$ and $b$. Surprisingly, for each pair $(a,b)$, only one of the two representations exists and the representation is also unique. We then investigate the representations of $(F_n-1)(F_{n+1}-1)/2$ and $(F_n-1)(F_{n+2}-1)/2$, where $F_i$ is the $i^{th}$ Fibonacci number, and observe several nice patterns. 

\pagestyle{myheadings} 
\thispagestyle{empty} 
\baselineskip=12.875pt 
\vskip 30pt
\section{Motivation and main results}
The $n$-th cyclotomic polynomial is defined as 
$$\Phi_n(x) \ =\ \prod_{m=1, (m,n) = 1}^n (x-e^{\frac{2\pi i m}{n}}).$$
Naturally, much work has been done on the values of the coefficients of $\Phi_n(x)$. Numbers of the form $(a-1)(b-1)/2$ with $(a,b) = 1$ arise in the study of the midterm coefficient of the $pq$-th cyclotomic polynomial, where $p, q$ are distinct primes. (Note that the degree of $\Phi_{pq}(x)$ is $\phi(pq) = (p-1)(q-1)$, where $\phi$ is the Euler totient function, so its midterm coefficient has degree $(p-1)(q-1)/2$.) In particular, these polynomials have been fully characterized by the work of Beiter \cite{Be}, Carlitz \cite{Ca} and Lam and Leung \cite{LL}. These authors used different and very clever approaches. 

In computing the midterm coefficient of $\Phi_{pq}(x)$, Beiter sketched a proof that $(p-1)(q-1)/2$ can be uniquely written as $\alpha q + \beta p + \delta$, where $0\le \alpha\le p-1$, $\beta \ge 0$ and $\delta \in \{0,1\}$. In this article, we provide an alternate proof of the result applied to any relatively prime numbers.

\begin{thm}\label{mainTheo}
Let $a, b\in\mathbb{N}$ be relatively prime. Consider two following equations.
\begin{align} \label{m1}xa+yb &\ =\ \frac{(a-1)(b-1)}{2}.\\
\label{m2}xa+yb+1&\ =\ \frac{(a-1)(b-1)}{2} .\end{align}
Exactly one of the two equations has nonnegative integral solution(s) and the solution is unique.  
\end{thm}

\begin{exa}\normalfont
We observe that both representations of $(a-1)(b-1)/2$ can happen. If $a = 3$ and $b=5$, we have $1\cdot 3 + 0\cdot 5 + 1 = (3-1)(5-1)/2$. If $a = 11$ and $b = 31$, we have
$8\cdot 11 + 2\cdot 31 = (11-1)(31-1)/2$. Our theorem is also related to Problem E1637 \cite{Mo}, which states that for $k\ge (a-1)(b-1)$, there exist nonnegative solution(s) to $xa+yb = k$. Our theorem gives examples of $k$ smaller than $(a-1)(b-1)$, which still make $xa + yb = k$ have a unique nonnegative solution. 
\end{exa}

\begin{cor}\label{Beiterim}
Let $p,q$ be distinct primes. Then
$(p-1)(q-1)/2$ can be uniquely written as $\alpha q+\beta p + \delta$ for some $0\le \alpha\le p-1$, $\beta\ge 0$ and $\delta \in \{0,1\}$.
\end{cor}

This corollary is what Beiter used in computing the midterm coefficient of $\Phi_{pq}(x)$. We now present another proof, which is shorter with the use of a strong theorem of Dresden that was not available when \cite{Be} first appeared. 
\begin{proof}[Alternate proof of Corollary \ref{Beiterim}]
By \cite[Theorem 1]{Dresden}, the midterm coefficient of $\Phi_{pq}(x)$ is odd. By \cite[Theorem 1]{Be}, $(p-1)(q-1)/2 = \alpha q+ \beta p + \delta$ in exactly one way.  
\end{proof}

Next, given a pair of consecutive Fibonacci numbers $(F_n, F_{n+1})$, we investigate the nonnegative, integral solutions to \begin{align}\label{1}F_n x + F_{n+1}y \ =\ (F_n-1)(F_{n+1}-1)/2\\
\label{2}1+F_n x + F_{n+1}y \ =\ (F_n-1)(F_{n+1}-1)/2\end{align}
Due to Theorem \ref{mainTheo} and $(F_n, F_{n+1}) = 1$\footnote{due to the Cassini's identity: $F_{n-1}F_{n+1}-F_n^2=(-1)^n$.}, we know that exactly one of these equations has a unique nonnegative, integral solution. By convention, we index the Fibonacci sequence as follows $$F_1 = 1,\mbox{ } F_2 = 1,\mbox{ } F_3 = 2,\mbox{ } F_4= 3,\mbox{ } F_5 = 5, \ldots .$$
The following table provides the first several cases.
\begin{center}
\begin{tabular}{ |c|c|c|c|c|c|c| }
\hline
$n$ & $F_n$ & $F_{n+1}$ & Which equation & $x$ & $y$\\
\hline 
$3$ & $2$ & $3$ & \eqref{2} & $0$ & $0$\\
$4$ & $3$ & $5$ & \eqref{2} & $1$ & $0$\\
$5$ & $5$ & $8$ & \eqref{2} & $1$ & $1$\\
\hline 
$6$ & $8$ & $13$ & \eqref{1} & $2$ & $2$\\
$7$ & $13$ & $21$ & \eqref{1} & $6$ & $2$\\
$8$ & $21$ & $34$ & \eqref{1} & $6$ & $6$\\
\hline
$9$ & $34$ & $55$ & \eqref{2} & $10$ & $10$\\
$10$ & $55$ & $89$ & \eqref{2} & $27$ & $10$\\
$11$ & $89$ & $144$ & \eqref{2} & $27$ & $27$\\
\hline
$12$ & $144$ & $233$ & \eqref{1} & $44$ & $44$\\
$13$ & $233$ & $377$ & \eqref{1} & $116$ & $44$\\
$14$ & $377$ & $610$ & \eqref{1} & $116$ & $116$\\
 \hline
\end{tabular}
\end{center}
We observe two patterns here. First, Equation \eqref{1} and Equation \eqref{2} are used alternatively with period $3$. Second, the first and the third row of each period give $x=y$. The two patterns are summarized by the following theorem.

\begin{thm}\label{fi}
For $k\ge 1$, the following formulas are correct. 
\begin{align}
    \label{first}\frac{1}{2}(F_{6k-1}-1)F_{6k} + \frac{1}{2}(F_{6k-1}-1)F_{6k+1} &\ =\ \frac{(F_{6k}-1)(F_{6k+1}-1)}{2}.\\
    \label{second}\frac{1}{2}(F_{6k+1}-1)F_{6k+1} + \frac{1}{2}(F_{6k-1}-1)F_{6k+2} &\ =\ \frac{(F_{6k+1}-1)(F_{6k+2}-1)}{2}.\\
    \frac{1}{2}(F_{6k+1}-1)F_{6k+2} + \frac{1}{2}(F_{6k+1}-1)F_{6k+3} &\ =\ \frac{(F_{6k+2}-1)(F_{6k+3}-1)}{2}.\\
    \label{fourth}1+\frac{1}{2}(F_{6k+2}-1)F_{6k+3} + \frac{1}{2}(F_{6k+2}-1)F_{6k+4} &\ =\ \frac{(F_{6k+3}-1)(F_{6k+4}-1)}{2}.\\
    1+\frac{1}{2}(F_{6k+4}-1)F_{6k+4} + \frac{1}{2}(F_{6k+2}-1)F_{6k+5} &\ =\ \frac{(F_{6k+4}-1)(F_{6k+5}-1)}{2}.\\
    1+\frac{1}{2}(F_{6k+4}-1)F_{6k+5} + \frac{1}{2}(F_{6k+4}-1)F_{6k+6} &\ =\ \frac{(F_{6k+5}-1)(F_{6k+6}-1)}{2}.
\end{align}
\end{thm}
\begin{rek}\normalfont
If $n$ is a multiple of $3$, then $F_n$ is even \cite{SE}. Let $k\in\mathbb{N}$. Because $(F_{3k}, F_{3k+1}) = (F_{3k+2}, F_{3k+3}) = 1$, $F_{3k+1}$ and $F_{3k+2}$ are odd. Hence, $F_n$ is even if and only if $n$ is a multiple of $3$. Therefore, if $n\not\equiv 0\Mod 3$, $\frac{1}{2}(F_n-1)$ is a nonnegative integer. 
\end{rek}

Similarly, given $(F_n, F_{n+2})$, we can consider two following equations
\begin{align}\label{10}F_n x + F_{n+2}y \ =\ (F_n-1)(F_{n+2}-1)/2\\
\label{20}1+F_n x + F_{n+2}y \ =\ (F_n-1)(F_{n+2}-1)/2\end{align}
Due to Theorem \ref{mainTheo} and $(F_n, F_{n+2}) = 1$\footnote{due to the identity: $F_n^2 - F_{n-2}F_{n+2}=(-1)^n$ \cite{LL2}.}, we know that exactly one of these equations has a unique nonnegative, integral solution. The following table provides the first several cases.
\begin{center}
\begin{tabular}{ |c|c|c|c|c|c|c| }
\hline
$n$ & $F_n$ & $F_{n+2}$ & Which equation & $x$ & $y$\\
\hline
$1$ & $1$ & $2$ & \eqref{10} & $0$ & $0$\\
$2$ & $1$ & $3$ & \eqref{10} & $0$ & $0$\\
$3$ & $2$ & $5$ & \eqref{10} & $1$ & $0$\\
\hline
$4$ & $3$ & $8$ & \eqref{20} & $2$ & $0$\\
$5$ & $5$ & $13$ & \eqref{20} & $2$ & $1$\\
$6$ & $8$ & $21$ & \eqref{20} & $6$ & $1$\\
\hline
$7$ & $13$ & $34$ & \eqref{10} & $10$ & $2$\\
$8$ & $21$ & $55$ & \eqref{10} & $10$ & $6$\\
$9$ & $34$ & $89$ & \eqref{10} & $27$ & $6$\\
\hline
$10$ & $55$ & $144$ & \eqref{20} & $44$ & $10$\\
$11$ & $89$ & $233$ & \eqref{20} & $44$ & $27$\\
$12$ & $144$ & $377$ & \eqref{20} & $116$ & $27$\\
 \hline
\end{tabular}
\end{center}
Again, Equation \eqref{10} and Equation \eqref{20} appear alternately with period $3$. The following theorem is compatible to Theorem \ref{fi} and the proof is similar, so we omit it. 
\begin{thm}\label{fi2}
For $k\ge 0$, the following formulas are correct. 
\begin{align}
    \frac{F_{6k+2}-1}{2}F_{6k+1} + \frac{F_{6k-1}-1}{2}F_{6k+3} &\ =\ \frac{(F_{6k+1}-1)(F_{6k+3}-1)}{2}\\
    \frac{F_{6k+2}-1}{2}F_{6k+2}+\frac{F_{6k+1}-1}{2}F_{6k+4} &\ =\ \frac{(F_{6k+2}-1)(F_{6k+4}-1)}{2}\\
    \frac{F_{6k+4}-1}{2}F_{6k+3} + \frac{F_{6k+1}-1}{2}F_{6k+5}&\ =\ \frac{(F_{6k+3}-1)(F_{6k+5}-1)}{2}\\
    1+ \frac{F_{6k+5}-1}{2}F_{6k+4} + \frac{F_{6k+2}-1}{2}F_{6k+6} &\ =\ \frac{(F_{6k+4}-1)(F_{6k+6}-1)}{2}\\
    1+ \frac{F_{6k+5}-1}{2}F_{6k+5} + \frac{F_{6k+4}-1}{2}F_{6k+7}&\ =\ \frac{(F_{6k+5}-1)(F_{6k+7}-1)}{2}\\
    1+\frac{F_{6k+1}-1}{2}F_{6k} + \frac{F_{6k-2}-1}{2}F_{6k+2}&\ =\ \frac{(F_{6k}-1)(F_{6k+2}-1)}{2}. 
\end{align}
\end{thm}
\section{Proofs}
We first prove the following lemma.
\begin{lem}\label{sg}
For any integers $n, x, y, a, b$ with $a, b$ positive and $(a,b) = 1$, we consider the equation $xa + yb = n$. If there is a solution $r, s$ with $r<b$ and $s<0$, then there are no solutions with $x, y$ nonnegative. 
\end{lem}
\begin{proof}
All solutions are of the form $(x, y) = (r + tb, s - ta)$ for some $t\in\mathbb{Z}$. To get $y\ge 0$, we must have $t<0$, but then $x< 0$. 
\end{proof}
\begin{proof}[Proof of Theorem \ref{mainTheo}]
Let $k = (a-1)(b-1)/2$. 

Let $0\le r_1\le b-1$ be chosen such that $ar_1\equiv k\Mod b$ and $s_1 := (k-ar_1)/b$. 

Let $0\le r_2\le b-1$ be chosen such that $ar_2\equiv (k-1)\Mod b$ and $s_2: = (k-1-ar_2)/b$.

Observe that $$a(r_1+r_2) \ =\ ar_1+ar_2\ \equiv\ 2k-1\ =\ a(b-1)-b\ \equiv \ a(b-1)\Mod b.$$
So, $b\ |\ (r_1+r_2 - (b-1))$ and so, $r_1+r_2 = b-1$. We compute 
\begin{align*}
    s_1 + s_2 \ =\ \frac{k-ar_1+k-ar_2-1}{b}\ =\ \frac{2k - a(b-1) - 1}{b}\ =\ -1.
\end{align*}
Hence, exactly one of $s_1, s_2$ is nonnegative. By definition, $r_1a+s_1b = k$ and $r_2a+s_2b +1 = k$. Therefore, we have shown that either Equation \eqref{m1} or Equation \eqref{m2} has a nonnegative solution $(x,y)$, while the other equation has no nonnegative solutions due to Lemma \ref{sg}.

It remains to prove that Equation \eqref{m1} has at most one nonnegative solution. (Similar claim and proof hold for Equation \eqref{m2}.) Let $(x_1,y_1)$ and $(x_2, y_2)$ be two nonnegative solutions of Equation \eqref{m1}. Clearly, $0\le x_1, x_2\le (b-1)$, so $|x_1-x_2|\le b-1$. Because $x_1a + y_1b = x_2a+y_2b$, $(x_1-x_2)a = -(y_1-y_2)b$. Since $(a,b) = 1$, $b$ divides $x_1-x_2$, which, along with $|x_1-x_2|\le b-1$, implies that $x_1=x_2$. It follows that $y_1=y_2$. Hence, Equation \eqref{m1} has at most one nonnegative solution. 
\end{proof}

\begin{proof}[Proof of Theorem \ref{fi}]
We prove Formulas \eqref{first}, \eqref{second} and \eqref{fourth}. Proofs for the remaining formulas are similar. 
We use Identity $(d7)$ \cite{LL2} repeatedly
\begin{align}\label{im}F_nF_{n+3}\ =\ F_{n+1}F_{n+2}+(-1)^{n-1}.\end{align}

Write
\begin{align*}
    &\frac{1}{2}(F_{6k-1}-1)F_{6k}+\frac{1}{2}(F_{6k-1}-1)F_{6k+1}\\
    \ =\ &\frac{1}{2}(F_{6k-1}-1)(F_{6k}+F_{6k+1})\\
    \ =\ &\frac{1}{2}F_{6k-1}F_{6k+2}-\frac{1}{2}(F_{6k}+F_{6k+1})\\
    \ =\ &\frac{F_{6k}F_{6k+1}+1}{2} - \frac{1}{2}(F_{6k}+F_{6k+1}) \mbox{ due to }\eqref{im}\\
    \ =\ &\frac{(F_{6k}-1)(F_{6k+1}-1)}{2}.
\end{align*}
This proves Formula \eqref{first}.

Next, we have
\begin{align*}
    &\frac{1}{2}(F_{6k+1}-1)F_{6k+1}+\frac{1}{2}(F_{6k-1}-1)F_{6k+2}\\
    \ =\ &\frac{1}{2}F_{6k+1}^2 - \frac{1}{2}F_{6k+1} + \frac{1}{2}F_{6k-1}F_{6k+2}-\frac{1}{2}F_{6k+2}\\
    \ =\ &\frac{1}{2}F_{6k+1}^2 - \frac{1}{2}F_{6k+1}+\frac{1}{2}(F_{6k}F_{6k+1} + 1) - \frac{1}{2}F_{6k+2}\mbox{ due to }\eqref{im}\\
    \ =\ &\frac{1}{2}F_{6k+1}^2 - \frac{1}{2}F_{6k+1}+\frac{1}{2}((F_{6k+2}-F_{6k+1})F_{6k+1}+1) - \frac{1}{2}F_{6k+2}\\
    \ =\ &\frac{1}{2}F_{6k+1}F_{6k+2} - \frac{1}{2}(F_{6k+1}+F_{6k+2}) + \frac{1}{2} \ =\ \frac{(F_{6k+1}-1)(F_{6k+2}-1)}{2}
\end{align*}
This proves Formula \eqref{second}

Finally, we prove Formula \eqref{fourth}. The left side is
\begin{align*}
    &1+\frac{1}{2}(F_{6k+2}-1)F_{6k+3} + \frac{1}{2}(F_{6k+2}-1)F_{6k+4}\\
    \ =\ &1+\frac{1}{2}F_{6k+2}F_{6k+5}-\frac{1}{2}(F_{6k+3}+F_{6k+4})\\
    \ =\ &1+\frac{1}{2}(F_{6k+3}F_{6k+4}-1)-\frac{1}{2}(F_{6k+3}+F_{6k+4})\\
    \ =\ &\frac{(F_{6k+3}-1)(F_{6k+4}-1)}{2},
\end{align*}
which is the right side. 
\end{proof}

\end{document}